\documentclass{article}
\usepackage{amsmath}
\usepackage{amsthm}

\usepackage{fullpage}
\usepackage{graphicx}
\usepackage{epstopdf}
\usepackage{hyperref}
\usepackage{amsrefs}
\usepackage[english]{babel}  % force American English hyphenation patterns
\usepackage{wrapfig}
\newcommand{\Aeq}{\textbf{Aeq }}
\newcommand{\Peq}[1][(h,ht)]{Peq$#1$ }
\newcommand{\Quads}[1][(T)]{Quads$#1$}

\newtheorem{theorem}{Theorem}
\newtheorem{lemma}{Lemma}

\newtheorem{question}{Question:}
\newtheorem{definition}{Definition:}
\newtheorem{conjecture}{Conjecture:}
\begin{document}

{\title{Revisiting the quadrisection problem of Jacob Bernoulli.}}
{\author{Carl Eberhart}}
\begin{abstract}
Two perpendicular segments which divide a given triangle into 4 regions of equal area is called a \textbf{quadrisection} of the triangle.
Leonhard Euler proved in 1779 that every scalene triangle has a quadrisection with its triangular part on the middle leg.  We provide a complete description of the quadrisections of a triangle. For example, there is only one isosceles triangle which has exactly two quadrisections.
\end{abstract}
%\subclass[2000]{01A99,08-03,51-03,51M04,51M15}
%\address{University of Kentucky, carl.eberhart@uky.edu}
\maketitle

\section{Introduction}

In 1687, Jacob Bernoulli~\cite{Bernoulli} published his solution to the problem of finding two perpendicular lines which divide a given triangle into four equal areas.  He gave a general algebraic solution which required finding a root of a polynomial of degree 8  and worked this out numerically for one scalene triangle.

%The 12 page paper \textit{No. XXIX. Solutio algebraica problematis de %quadrisectione trianguli scaleni, per duas normales rectas} can be found %in vol 1 of his collected works:

%\noindent\url{http://e-rara.ch/zut/content/pageview/1278319}.

The question of whether Bernoulli's  polynomial equation of degree 8 has the needed root in all cases is not answered completely. Leonhard
 Euler's~\cite{Euler1}, in 1779, wrote a 22 page paper in which he gives a complete solution using trigonometry.

% \textit{Dilucidationes Problemate Geometrico De Quadrisecione Trianguli %a Jacobo Bernoulli.} can be found online at %\url{http://eulerarchive.maa.org}, Index Number E729.

\begin{wrapfigure}{R}{0.5\textwidth}
\centering
\includegraphics[scale=.4]{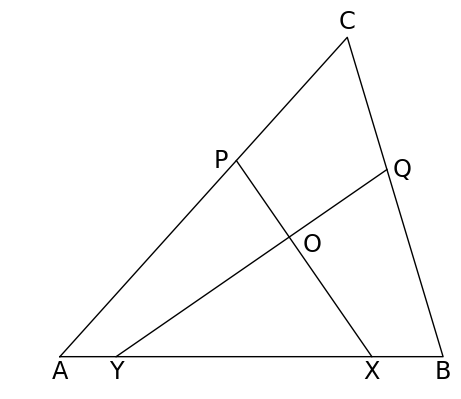}
%\caption{\label{fig:ber1}A.}

\end{wrapfigure}

Euler states his solution in a theorem which we paraphrase.

\begin{theorem}(Euler 1779)
Given a scalene triangle $\Delta ABC$ with  $AB$ the side of middle length, there is a quadrisection $XP$ and $YQ$ intersecting in a point $O$ in the interior of the triangle so that $X$ and $Y$ lie on side $AB$ and triangle $XOY$ is one of the $4$ areas of the quadrisection. The other areas of the quadrisection are quadrangles.
\end{theorem}

Euler does not claim that the triangular portion of a quadrisection must lie on the side of middle length. Also, he does not appear to discuss whether there is more than one quadrisection of a triangle, except to note that an equilateral triangle has 3 quadrisections.  In fact, we will see there are lots of triangles with quadrisections  where the triangular portion lies on the shortest side, but no triangles having a quadrisection  with the triangular portion on the longest side.

This paper was written in latex using an account  on cloud.sagemath.com.  Some animations designed to augment this paper can be found at
\url{http://www.ms.uky.edu/~carl/sagelets/arcsoftriangles.html}

\section{Initial analysis.}
Take any triangle $T$.  We can scale and position it in the plane so that two vertices $A$ and $B$ have coordinates $(0,0)$ and $(1,0)$ respectively and the third vertex $C=(h,ht)$ is chosen so that it is in the quadrant $h\ge 1/2$ and $ht > 0$. Under these assumptions, there is only one choice for $C$,
$(1/2,\sqrt{3}/2)$,  if $T$ is the equilateral triangle.

If $T$ is an isoceles triangle, then there are two possibilites:
(1) if the vertex angle is greater than $\pi/3$, then  $C=(1/2,ht)$ with
$ht<\sqrt{3}/2$ or $C=(h,\sqrt{2\,h - h^2})$ with $1/2 < h < 2$,
(2) if the vertex angle is less than $\pi/3$, then  $C=(1/2,ht)$ with $ht > \sqrt{3}/2$  or $C=(h,\sqrt{1-h^2})$ with
$1/2 < h <1$.

If $T$ is a scalene triangle, then depending on how we chose $AB$, $C$ will be in one of three open regions $R_1,R_2,R_3$ respectively:

(1)  If $AB$ is the longest leg, then then $h^2+ht^2 < 1$,

(2)  If $AB$ is the middle leg, then $h^2+ht^2>1$ and $(h-1)^2+ht^2 <1$, or

(3) If $AB$ is the shortest leg, then $(h-1)^2+ht^2>1$

Since $AB$ can be any one of the 3 sides of $T$, we know each of the three regions above together with its boundary of isosceles triangles contains a unique copy of each triangle up to similarity. Denote these sets by $\overline{R_1},\overline{R_2},\overline{R_3}$ respectively.

We can use inversion about a
circle\footnote{
Regarding points as vectors, the \textbf{inverse} of $p=(x,y)$ about the circle of radius 1 centered at $c=(h,k)$ is $p'=c+(p-c)/(p-c)^2=(h+(x-h)/d,k+(y-k)/d)$, where $d=(x-h)^2+(y-k)^2$.
\url{https://en.wikipedia.org/wiki/Inversive_geometry}}
to match up a triangle in one region with its similar versions in the other regions using inversion about a circle. Inversion about the unit circle interchanges  points  in region 2 with points  in region 1, and inversion about the unit circle centered at $(1,0)$ interchanges points in region 2 with points  in the unbounded region 3. For example, let $C$ be a point in region 2, and let $C^{'}$ be it's inverse about the unit circle centered at $(1,0)$.
Then $C^{'}$ is in region 1, and $B,C,C^{'}$ are collinear with $C$ between $B$ and
$C^{'}$.  Further $BC\,BC^{'} = 1$.  Using this, we see that $\Delta ABC$ is similar with $\Delta C^{'}BA$. By the same reasoning, letting $C^{*}$ denote the inverse of $C$ about the unit circle, we get that $\Delta ABC$ is similar to $\Delta AC^{*}B$.

\smallskip
\textbf{Terminology:} In what follows, when we refer to a \underline{triangle} $C=(h,ht)$, we are referring to $\Delta ABC$ with $A=(0,0),\,B=(1,0)$, where $h\ge 1/2$ and $ht >0$.

\begin{center}
\includegraphics[scale=.6]{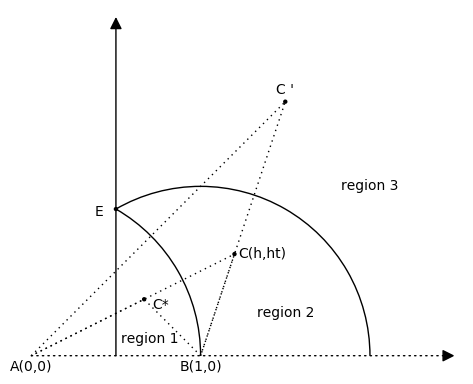}
\end{center}

\section{The equations for quadrisection of a triangle.}

Regarding points as vectors, we write $X=x\,B$, $Y=(1-y)\,B$, $P=s\,C$, and $Q=(1-r)B+r\,C$ for $x,y,r,s \in [0,1]$. Now  since  $\text{ area }\Delta XPY = s/2\, ht\, x = \text{ area }\Delta ABC =  ht/2$, we see that  $s=1/(2 x)$. Similarly, $r=1/(2 y)$.
Note that $1/(2 x)<1$, so $1/2<x$ and similarly for $y$.  So $1<x+y$.
\par

There are two equations which determine $x$ and $y$:

\noindent\textbf{1. The area of triangle $\Delta XOY$ is one fourth of the total area of the triangle.}

Writing $O=(x_0,y_0)$, this is the area equation  $4\,y_0(x+y-1)=ht$.  We can calculate $y_0$ by writing $O=u\,P+(1-u)\,X=v\,Q-(1-v)\,Y$ for some $u,v \in [0,1]$.  Expand this out to get two linear equations in $u,v$.  Solve for $u$ to get $u=\dfrac{( x +  y)-1}{   (x + \dfrac{s}{r} y)-s}$.  Substitute this into $O=u\,P+(1-u)\,X)$ and calculate
$O_1=y_0=\mathit{ht}\,\dfrac{1-( x +  y ) }{ 1 -  \dfrac{x}{s} - \dfrac{y}{r}}=ht\,\dfrac{1-(x+y)}{1-2\,(x^2+y^2)}$. Substitute this into the area equation, divide both sides by $ht$ and simplify to get the

\begin{center}
\textbf{Area Equation: (Aeq)}$\ \ (x^2+y^2)+4\,(x\,y-x-y)+5/2=0$
\end{center}

This equation has two solutions for $y=y(x)$ in terms of $x$, but the one we want is  $y(x)=2-2 x+\sqrt{12 x^2 -16 x +6}/2$.  Note $y(\sqrt{2}/2)=1$ and $y(1)=\sqrt{2}/2$.
Also $y(5/6)=5/6$. \textbf{Note:} In the duration, $y(x)=2-2 x+\sqrt{12 x^2 -16 x +6}/2$.

\noindent\textbf{2. $XP$ and $YQ$ are perpendicular}. This means  the dot product $(Q-Y) \cdot(P-X) = 0$, or
$\left(h\,s - x,\,\mathit{ht}\, s\right)\cdot \left(h\, r - r + y,\,\mathit{ht}\, r\right) = 0$.

Substitute $s=1/(2\,x)$, $r=1/(2\,y)$ and simplify to get the

\begin{center}
\textbf{Perpendicularity Equation: (Peq)}  ${\ \ \left( x^2 - h/2\right)} {\left(y^2 - (1-h)/2\right) = \left(ht/2\right)^{2}}$
\end{center}

To use \Peq to find all the quadrisections of a particular triangle $T$, substitute $y=y(x)$ into \Peq and solve the resulting equation in $x$ for each $(h,ht)$ (with $h\ge 1/2, ht > 0$) which is similar to $T$. The total number of solutions is the number of quadrisections.

\section{ \textbf{Peq} viewed as a  1-parameter family of circular arcs}

If instead of setting the values for $h$ and $ht$ in \Peq,  set the value of $x$ between $\sqrt{2}/2$ and $1$, and then set $y=y(x)$, then \Aeq is satisfied and we have fixed the base $YX$ on $AB$ of the triangular part of a quadrisection.  Then $\Peq[(x,y(x))]$ is a quadratic equation in $h$ and $ht$, which if we rewrite in the form

$\Peq[(x,y(x)):]\ \   ht^2 + \left(h-(x^2-y(x)^2+1/2)\right)^2= \left(x^2+y(x)^2-1/\right)^2$

\noindent we recognize as the circle $Cir(x)$ in the $h,ht$ plane with center $X(x)=(c(x),0)$ and radius $r(x)=x^2+y(x)^2-1/2$, where $c(x)=x^2-y(x)^2+1/2$.

The arc $Arc(x)$ of the circle $Cir(x)$ that lies in the quadrant $h\ge 1/2, ht>0$ consists of all triangles $(h,ht)$ with a quadrisection with base $YX=[(1-y(x),0),(x,0)]$.  $Arc(\sqrt{2}/2)$ lies in the unit circle, and forms the lower boundary of Region 2, and $Arc(1)$ lies in the unit circle centered at $(1,0)$ and forms the upper boundary of Region 2.  Let $(1/2,z(x))$ be the terminal point on $Arc(x)$, and let $\theta(x)$ be the radian measure of $\angle (1/2,z(x))(c(x),0)(2,0)$. So $\theta(x)=\arccos((1/2-c(x))/r(x)) $, and $z(x)=\sqrt{r(x)^2-(1/2-c(x))^2}$.

Let \textbf{Arcs} denote the union of all arcs $Arc(x)$.  The triangles $(h,ht)\in\text{\textbf{Arcs}}$ are precisely the triangles which have a quadrisection with the triangular portion on $[(0,0),(1,0)]$.

\subsection{A useful mapping.}

\begin{wrapfigure}{R}{0.6\textwidth}
\centering
\textbf{A small table of values.\vspace{.05in}}
\begin{tabular}{| c | c | c|c|c|c|}
\hline
$x$ & $y(x)$& $c(x)$ & $r(x)$ & $\theta(x)$ & $T(x,\theta(x))=(1/2,z(x))$ \\
\hline
$\sqrt{2}/2$ & 1 & 0 & 1  &  $\pi/3$  & $(1/2,\sqrt{3}/2)$\\
\hline
$5/6$ & $5/6$ & $1/2$ & $8/9$ & $\pi/2$ & $(1/2,8/9)$\\
\hline
$1$  & $\sqrt{2}/2$ & $1$ & $1$ & $2\pi/3$ & $(1/2,\sqrt{3}/2)$\\
\hline
\end {tabular}\vspace{-.05in}

\end{wrapfigure}

  Let

$D=\{(x,\theta)|x\in [\sqrt{2}/2,1], \theta\in [0,\theta(x) ]\}$.

\noindent We define a mapping $F$ from $D$ onto \textbf{Arcs} by

$F(x,\theta)=(c(x)+r(x)\cos(\theta),r(x)\sin(\theta))$.

$F$ maps each segment $[(x,0),(x,\theta(x))]$ in $D$ 1-1 onto the corresponding arc $Arc(x)$.

The Jacobian determinant of $F$ is $|J_F(x,\theta)|=c'(x) r(x) \cos(\theta)+r'(x)r(x)$.   This vanishes on the curve $J_0=\{ p(x)| x\in [\sqrt{2}/2,1]\}$, where $p(x)=(x,\arccos(-r'(x)/c'(x))$.  $D \setminus J_0$ is the union of two disjoint relatively open sets $U$, and $V$ in $D$, with $(\sqrt{2}/2,0)\in U$.  Let $D_1=U \cup J_0$ and $D_2=V \cup J_0$. See the diagram.
$D_1$ and $D_2$ are both topological closed disks, with boundaries $\partial D_1 = S_1 \cup S_2 \cup S_3 \cup J_0 \cup C_1$, and  $\partial D_2 = J_0 \cup S_1 \cup C_2$  as shown in the diagram. It follows from the definition of $F$ that it is 1-1 on each of $D_1$ and $D_2$.

Note that $F(S_1 \cup S_2)=Arc(1)$, $F(S_4)=Arc(\sqrt{2}/2)$, $F(S_3)=[(1,0),(2,0)]$ (not corresponding to any triangles), and $F(C_1)=F(C_2)=[(1/2,\sqrt{3}/2),(1/2,8/9)]$.
$F(J_0)$ is the concave up portion of the upper boundary of $F(D_1)$.
Each arc $Arc(x)$ for $x\in [5/6,1]$ is tangent to it, so $F(J_0)$ is the envelope\footnote{\url{https://en.wikipedia.org/wiki/Envelope_(mathematics)}} of those arcs.

Let $R_4$ denote the closed disk with boundary $F(J_0) \cup F(C_1)\cup F(S_4) \cup F(S_3)\cup F(S_2)$, and $R_5$ the closed disk with boundary $F(J_0) \cup F(S_1) \cup F(C_1)$.  Then $R_4\supset \text{\textbf{Arcs}}$, and $R_5 \subset \overline{R_3}$ with $R_5\cap R_4=F(J_0)$

\begin{lemma}
The transformation $F$ maps $D_1$ homeomorphically onto the closed disk $R_4 \supset \overline{R_1}$, and  $D_2$ homeomorphically onto $R_5$.
\end{lemma}

\begin{proof}
$F$ maps the boundaries $\partial D_1$ and $\partial D_2$ onto the boundaries $R_4$ and $R_5$, respectively.
Consequently, it follows from the Brouwer invariance of domain theorem \footnote{\url{https://en.wikipedia.org/wiki/Invariance_of_domain}}, that $F(D_i)$ is the closed disk enclosed by $F(\partial D_i)$ for $i=1,2$.
\end{proof}

\begin{center}
\includegraphics[scale=.5]{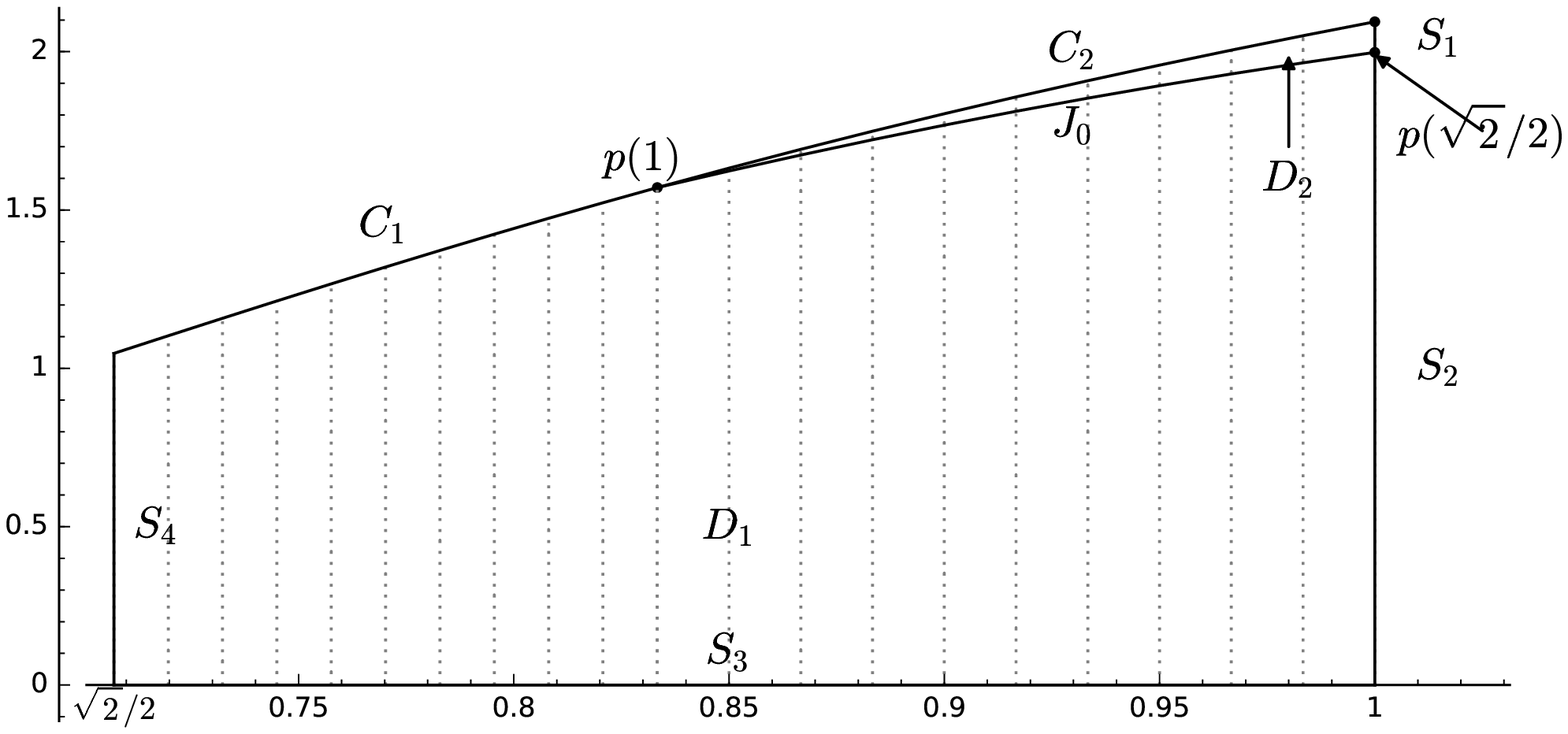}
\end{center}

Note that the orientation of $\partial D_2$ is the reverse of the orientation of $F(D_2)$.  This is because $|dF|<0$ on $D_2$.  $F$ folds $D_2$ over along $J_0$ and fits it onto $F(D_2)$.

\begin{center}
\includegraphics[scale=.5]{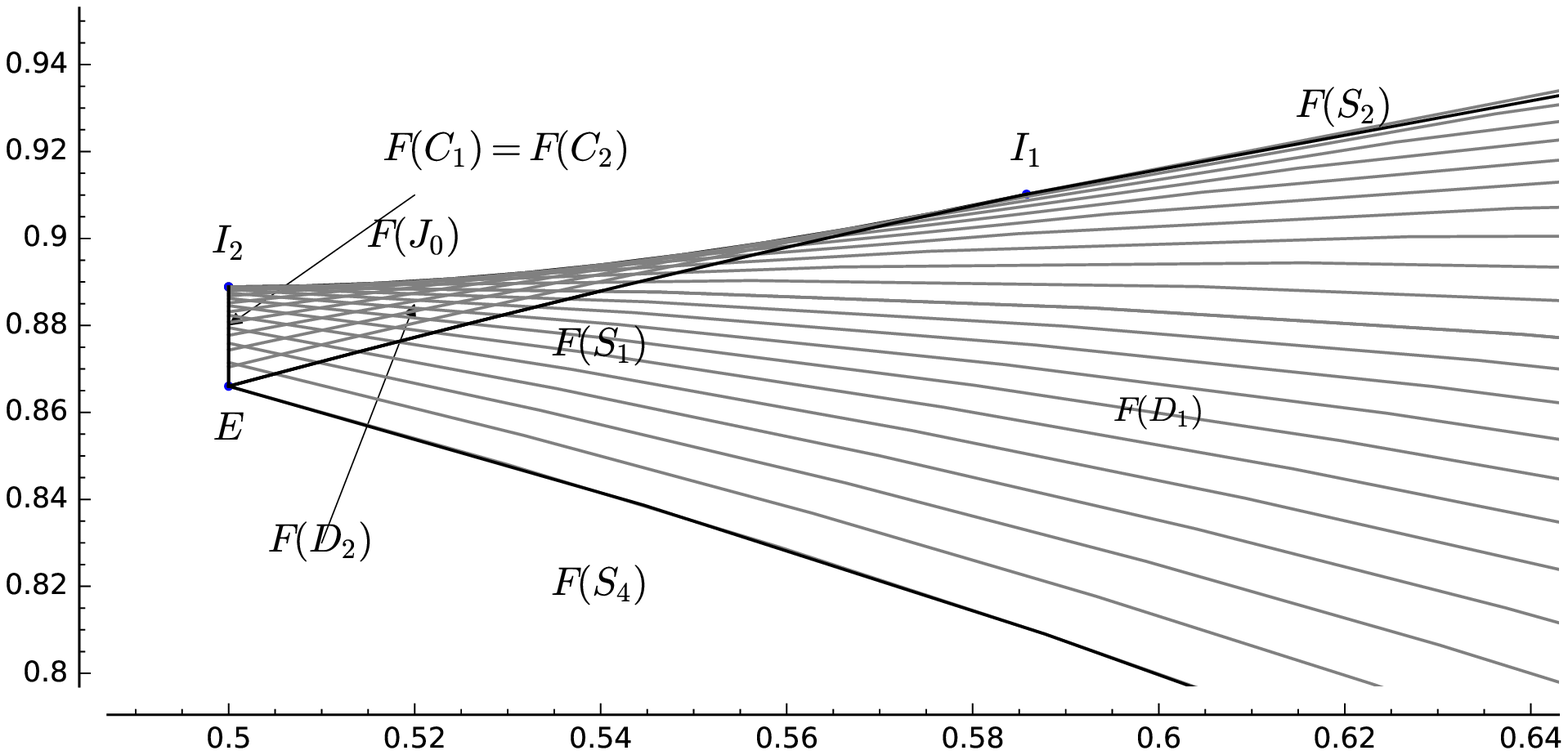}
\end{center}

Since $F(D_1)$ contains $R_1$, we have another proof of Euler's theorem. In fact, since $F(D_2)\cap R_1 = \emptyset$, we have a stronger version.

\begin{theorem}
Each scalene triangle $T$ has \textbf{exactly} one quadrisection whose triangular portion lies on the middle leg of $T$ and \textbf{no} quadrisection whose triangular portion lies on the longest leg of $T$.
\end{theorem}
\begin{proof}
Scale and position $T$ so that its vertices are $(0,0),(1,0),C$ with $C=(h,ht)\in R_1$. Then since $R_1 \subset F^(D_1)$, $(x,\theta)=F^{-1}(h,ht)$  gives a  quadrisection of $T$ with its triangular portion on its middle leg. But also $F(D_2)=R_4 \subset \overline{R_2}$, and $\overline{R_2}\cap R_1 = \emptyset$ so $T$ has only the 1 quadrisection with it's triangular portion on the middle leg.  and no quadrisection with its triangular portion on the shortest leg.
\end{proof}

\smallskip

\subsection{Counting the quadrisections of a triangle}

\begin{definition}  Let \Quads[(T)]  denote the number of quadrisections of the triangle $T$.  If $(h,ht)$ is a point with $h\ge 1/2, ht > 0$  such that $T$ is similar to $\Delta A(0,0)B(1,0)C(h,ht)$, then we write \Quads[(h,ht)]=\Quads[(T)].
\end{definition}

\begin{theorem}
\textbf{(Quadrisection theorem for triangles) }  Let $T$ be a triangle. Let $C=(h,ht)$ be the unique triangle in $\overline{R_2}$  similar to $T$.
Let $C'=(h',ht')$ be the inverse of $C$ about the unit circle with center $(1,0)$.

\begin{enumerate}

\item  If $C' \in F(D_2)\setminus F(J_0)$ then \Quads$=3$.

\item  If $C' \in F(J_0)$ and $C' \ne I_1$, then \Quads$=2$.

\item  Otherwise (that is, if $C'\notin F(D_2)$ or $C'=I_1$), then \Quads$=1$.
\end{enumerate}

\end{theorem}

\begin{proof}
Assume case 1. $F(D_2)\setminus F(J_0)$ is doubly covered, once by $D_2 \setminus J_0$ and once by $D_1$.  So $T$ has two quadrisections with the triangular part on a shortest side.  Since $T$ also has a quadrisection with the trianglular part on its middle side, \Quads$=3$.

Assume case 2.  $F(J_0 \setminus \{p(\sqrt{2}/2)\})$ is covered once by $F$.  So $T$ has a single quadrisection with the triangular portion on the shortest side.  Since $T$ also has a quadrisection with the trianglular part on its middle side, \Quads$=2$

Assume case 3.   If $C'\notin F(D_2)$ then $T$ has only the quadrisection with triangular part on a middle side.  If $C'=C=I_1=F(p(\sqrt{2}/2))$, then $Arc(1)$ is the only arc that contains $C'$.   So \Quads$=1$.
\end{proof}
The isoseceles triangles $I_1$ and $I_2$ occupy interesting positions amongst the isosceles triangles.  The vertex angle of $I_1$ is greater than $\pi/3$ and it has only 1 quadrisection.  Any other isosceles triangle with these 2 properties has a larger vertex angle.  $I_2$ is interesting because it is the only isosceles triangle with exactly 2 quadrisections. But also, one of the quadrisections is rational, that is, the vertices of the triangle and  the endpoints of the segments forming the quadrisection are rational.

\begin{center}
\includegraphics[scale=.4]{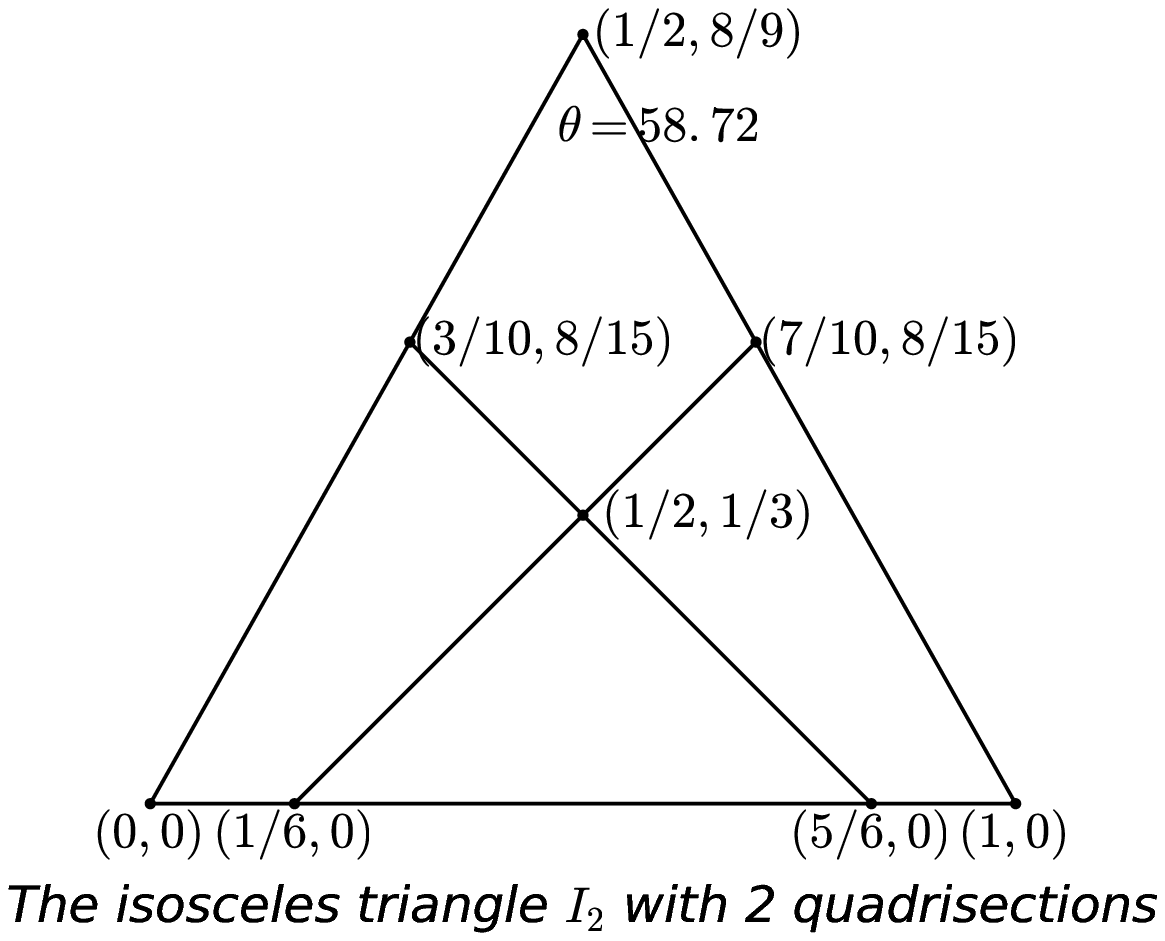}
\includegraphics[scale=.4]{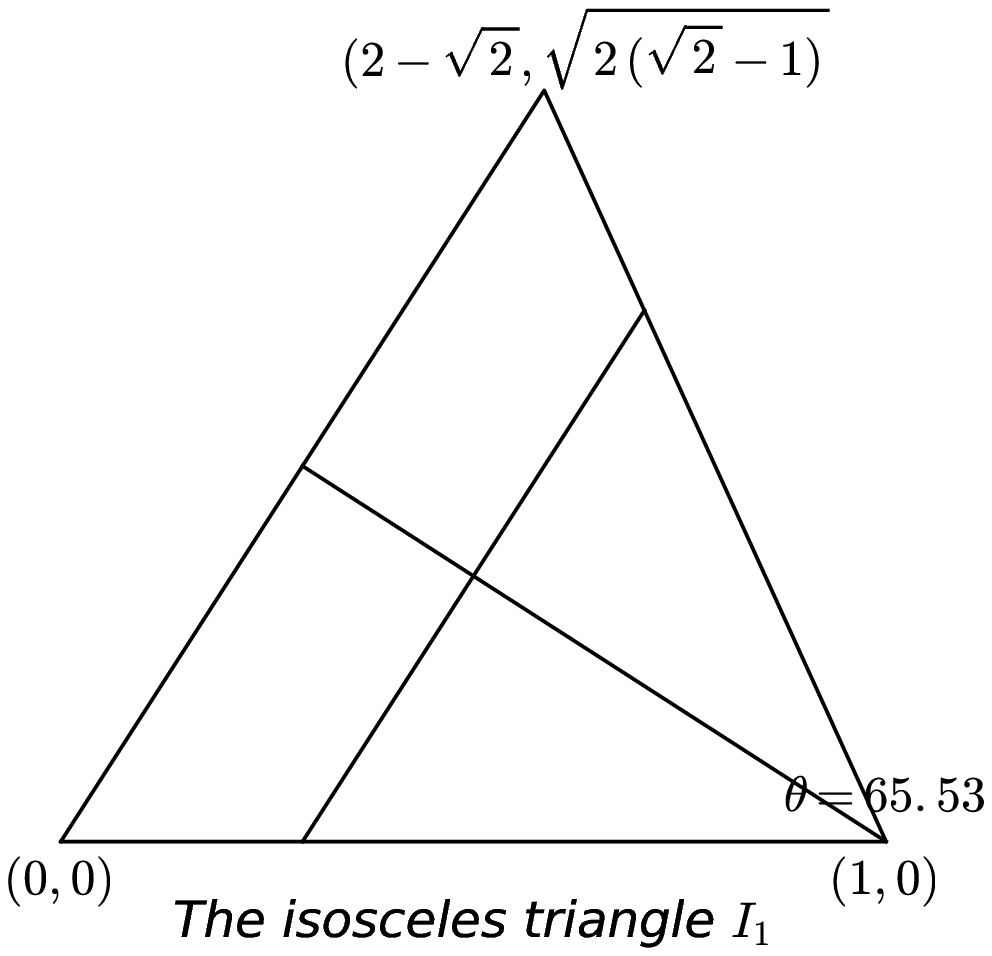}
\end{center}

\begin{question}  Is there another triangle with rational vertices and a rational quadrisection?
\end{question}

\section{The space of triangles}

Let $\Upsilon$ be the set $\{(h,ht)\ |\  1/2 \le h < 2,\  ht >0,\ \text{with}\ \sqrt{1-h^2}\le ht \le \sqrt{h(2-h)}\}$, with the subspace topology from the Cartesian plane.  As observed earlier, every triangle is similar to exactly one triangle  $[A(0,0),B(1,0),C(h,ht)]$ with $(h,ht)\in \Upsilon$. So it is natural to call $\Upsilon$ the space of triangles.  In that space, we see that the reflection of $T(J_0)$ about the circle of radius 1 centered at $(1,0)$ is an arc $S_2$ of scalene
triangles from $I_1$ to the reflection of $I_2$ which is $C=(175/337,288/337)$. All these have \Quads[(C)]$=2$, except \Quads[(I_2)]$=1$.  This arc separates $\Upsilon$ into two relatively open sets $U$ and $V$, with the equilateral triangle $E=(1/2,\sqrt{3}/2)\in U$. \Quads[(C)]$=3$ for each $C\in U$, and \Quads[(C)]$=1$ for each $C\in V$.

\begin{wrapfigure}{R}{0.6\textwidth}
\centering
\includegraphics[scale=.5]{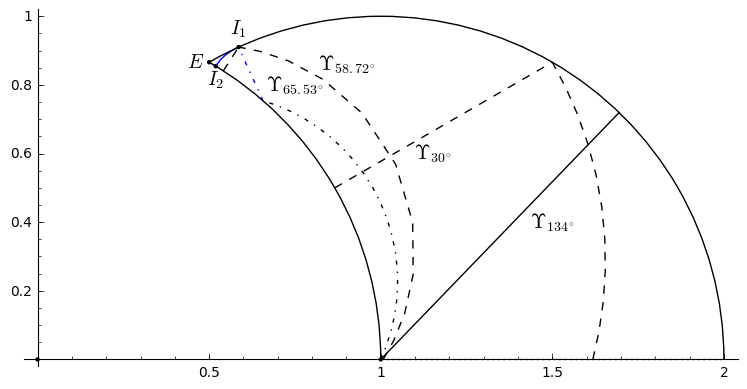}
\end{wrapfigure}
Note that the vast majority of triangles have only one quadrisection, and the only ones that have more than one quadrisection are very near the equilateral triangle.  It also gives a useful way to visualize the position of certain classes of triangles in the space, such as the isosceles triangles, which form the boundary of the space.  Where does the class $\Upsilon_{\pi/2}$ of right triangle sits in $\Upsilon$?  Since the points $(1,ht)$ with $ht>1$ lie in region 3, they invert onto the points $(1,1/ht)$, so
it is simply the vertical segment $((1,0),(1,1)]$ in $\Upsilon$. Topologically, it separates the space $\Upsilon$ into two pieces. The right hand piece consists of all triangles with an obtuse angle; the left hand piece all triangles with all angles acute.  Similarly, the class $\Upsilon_{\theta}$ for all $\theta$ with $\pi/2 \le \theta < \pi$ is a segment in $\Upsilon$ separating it into two pieces.   If $0 < \theta \le \pi/2$, then one works out that $\Upsilon_{\theta}$  consists of a radial segment in $\Upsilon$ together with a circular arc in $\Upsilon$ which separates $\Upsilon$ into three pieces.

There are two angles of particular interest:

\noindent 1) the vertex angle $\theta$ of $I_2$, which we have calculated as $\theta = 360/\pi \arctan(9/16)\approx 58.72^o$. Any triangle with an angle at least $\theta$ has only the one quadrisection.

\noindent 2) the vertex angle of $I_1$, which we have calculated as $\theta = 180/\pi \arctan(\sqrt(2\sqrt{2}-2))\approx 65.53^{o}$

\section{Further Questions.}

This half-lune model for the space of triangles is reminiscent of the Poincare model for the hyperbolic plane.
Since no two triangles are similar in the hyperbolic plane, there are more triangles to quadrisect. This suggests that the quadrisection problem  may be more delicate to solve completely in the hyperbolic plane.

\begin{question} What is the analogue for Theorem 3 in the hyperbolic plane?
\end{question}

Our discussion of the quadrisection problem for triangles naturally leads to the question of determining the quadrisections of any convex polygon.   Investigating this question leads to the following conjecture.

\begin{conjecture} A convex $2\,n+1$-gon $R$ has at most $2\,n+1$ quadrisections.
Further, if $R$ is 'sufficiently close to' the regular $2\,n+1$-gon, then $R$ has $2\,n+1$ quadrisections.
\end{conjecture}

\section{Historical notes}

\subsection{Bernoulli's solution}
Bernoulli worked in a time before the use of Cartesian geometry had become widespread, so it is not surprising that he did not coordinatise and normalize the problem.  In any case, he did obtain a method for constructing triangles and their quadrisections. However, it is not made clear that the construction produces all possible triangles.

\begin{center}
\includegraphics[scale=.6]{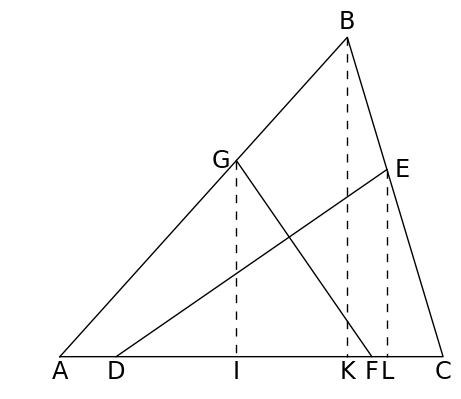}

\textbf{Bernoulli diagram}
\end{center}

Using Bernoulli's labeling of the triangle, let $AC=a$, $CB=b$, $BA=c$, $KB=d$, $KC=e$, $KA=f$, $CD=x$, $AF=y$. Note that Bernoulli's $x$ and $y$ are our $y$ and $x$.  Bernoulli derives versions of the area and perpendicularity equations:

\textbf{Bernoulli's equations:}
$y^2=4\,a\,y-4\,x\,y -2\,\frac{1}{2}\,a^2 + 4\,a\,x - x^2$ and
$y^2=\pm \frac{1}{2}a\,f+\dfrac{a^2\,d^2}{4\,x^2\pm 2\,a\,e}$.

If we normalize these equations by letting $a=1$, then $d=ht$, $f=h$, $e=1-h$, and Bernoulli's equations are our \Aeq and \Peq.  He then reduces his equation to a polynomial equation in one variable of degree 8.  If we  normalize by letting $a=1$, then the polynomial 'simplifies' to

\noindent $p(x)=x^{8}- 8  x^{7} + {(3  b^{2} - 3  c^{2} + 17)} x^{6}  - 2  {(b^{2} - c^{2} + 5)} x^{5}\\ - \frac{1}{4}  {(3  b^{4} - 6  b^{2} c^{2} + 3  c^{4} + 38  b^{2} - 24  c^{2} + 17)} x^{4}\\
+ (b^{4} - 2  b^{2} c^{2} + c^{4} + 12  b^{2} - 6  c^{2} + 5) x^{3}
+ \frac{1}{4}  {(4  b^{4} - 5  b^{2} c^{2} + c^{4} - 7  b^{2} - 1)} x^{2}\\
- \frac{1}{2}  {(4  b^{4} - 5  b^{2} c^{2} + c^{4} + 5  b^{2} - 2  c^{2} + 1)} x\\
+\frac{3}{4}  b^{4} - \frac{3}{4}  b^{2} c^{2} + \frac{3}{4} b^{2} - \frac{1}{8}  c^{2}+ \frac{1}{16}  c^{4} + \frac{1}{16}=0$.

Bernoulli describes a  method for constructing simultaneously a triangle and a quadrisection using the area and perpendicularity equations, and illustrates it by constructing a triangle with $a=484, b=490, c=495, x=386$.  Checking this with his normalized polynomial, we get $ p(386/484)\,484=2.85$, which is not $0$ but relatively close to $0$. The correct value rounded to two decimals for this $x$ is $x=368.86$.  This triangle is close enough to equilateral that it has 3 quadrisections.

\subsection{Euler's clarification}

Euler  chooses to use angles $\alpha, \beta, \gamma, \phi$ in his analysis.

\begin{center}
\includegraphics[scale=.6]{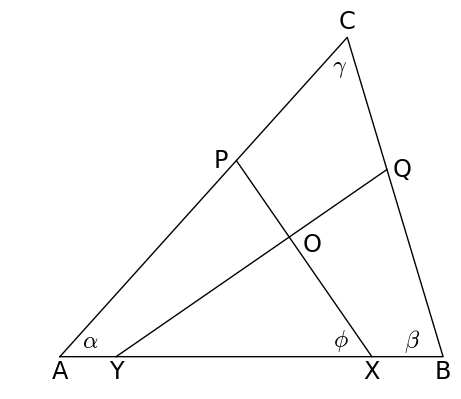}

\textbf{Euler's diagram}
\end{center}

Using the diagram, let $AX=x$, $YB=y$, $f=\cot \alpha, g=\cot \beta$,and $ t=\tan \phi$. Then he shows that $x=k\,\sqrt{f+1/t}$, $y=k\,\sqrt{g+t}$, where $k^2$ denotes the area of the triangle.  Next he obtains a single equation in the unknown $t$:

\textbf{Euler's equation}  $\sqrt{f+\dfrac{1}{t}}+\sqrt{g+t}-\sqrt{2\,(f+g)} = \sqrt{\dfrac{1+t^2}{2\,t}}$

Euler then makes use of the equation to construct a lengthy direct proof of his theorem that every scalene triangle has a quadrisection with its triangular part on the middle side.

He also shows how to use his equation to estimate the value $\phi$ to the nearest second for the right triangle with sides $2,1,\sqrt{5}$ and also calculates $x=1.5146$.   This right triangle, as with all right triangles, has only one quadrisection, and the correct value for $x$ rounded to 5 decimals is $x=1.51443$, so his estimate is pretty close.

\subsection{An explanation of interest.}
I became aware of the triangle quadrisection problem when looking up biographical information about Jacob Bernoulli in connection with a fictional story (\textit{An Elegant Solution}, by Paul Robertson, 2013) about a young Leonard Euler and the Bernoulli brothers, Johann and Jacob.   Several entries mentioned this as his contribution to geometry.  It is interesting that both mathematicians in this story wrote papers on this same subject.

I want to thank my colleagues Dr. Fred Halpern and Aaron Raden for helpful discussions during the writing of this paper.

\end{document}